\documentclass{scrartcl}
\usepackage[british]{babel}
\usepackage[textwidth=150mm,top=2cm]{geometry}

\usepackage{dsfont, amsthm, mathtools,  enumerate,amsmath,amsrefs}

\usepackage{dsfont}
\usepackage{amssymb}
\usepackage{aurical}
\usepackage{mathtools}
\usepackage{longtable}
\usepackage{wasysym}
\usepackage{euscript}
\usepackage{mathrsfs}
\usepackage{tikz}
\usepackage[hidelinks]{hyperref}

\usepackage{libertine}
\newcommand{\QR}[2]{
\raisebox{0.5ex}{\small\ensuremath{#1}}
\ensuremath{\mkern-3mu}\diagup\ensuremath{\mkern-3mu}
\raisebox{-0.75ex}{\small\ensuremath{#2}}}

\newcommand{\Mfg}{\text{\Fontlukas M}\ }

\newcommand{\I}{\mathfrak{T}}

\newcommand\bcdot{\ensuremath{%
  \mathchoice%
   {\mskip\thinmuskip\lower0.2ex\hbox{\scalebox{1.5}{$\cdot$}}\mskip\thinmuskip}}%
   {\mskip\thinmuskip\lower0.2ex\hbox{\scalebox{1.5}{$\cdot$}}\mskip\thinmuskip}%
   {\lower0.3ex\hbox{\scalebox{1.2}{$\cdot$}}}%
   {\lower0.3ex\hbox{\scalebox{1.2}{$\cdot$}}}%
}
\providecommand{\skalarProd}[2]{#1\!\bcdot\!#2}

\newcommand{\aus}{\ \! \raisebox{1pt}{\scalebox{1.3}{$\lrcorner$}} \ \!}
\newcommand{\ausR}{\ \! \raisebox{1pt}{\scalebox{1.3}{$\llcorner$}}\ \!}

\renewcommand{\P}{\mathcal{P}}
\renewcommand{\S}{\mathcal{S}}

\usepackage{tikz}
\newcommand{\tikzoverset}[2]{%
  \tikz[baseline=(X.base),inner sep=0pt,outer sep=0pt]{%
    \node[inner sep=0pt,outer sep=0pt] (X) {$#2$};
    \node[yshift=1pt] at (X.north) {\fontsize{6}{0}$#1$};
}}

\renewcommand{\H}{{\tikzoverset{\rightharpoonup}{\EuScript{H}}}}

\providecommand{\Stufe}[2]{\left<\vphantom{A}\right.\!#1\!\left.\vphantom{A}\right>_{#2}}
 \newtheorem{thm}{Theorem}[section]
 \newtheorem{cor}[thm]{Corollary}
 \newtheorem{lem}[thm]{Lemma}
 \newtheorem{prop}[thm]{Proposition}
 \theoremstyle{definition}
 \newtheorem{defn}[thm]{Definition}
 \theoremstyle{remark}
 \newtheorem{rem}[thm]{Remark}
 \newtheorem*{ex}{Example}
 \numberwithin{equation}{section}

\begin{document}


\title{Geometric Calculus of the Gauss Map}

\author{Peter \textsc{Lewintan}\footnote{peter.lewintan@uni-due.de, University of Duisburg-Essen, Germany}}

\date{October 5, 2016}

\maketitle

\begin{abstract}
In this paper we connect classical differential geometry with the concepts from geometric calculus. Moreover, we introduce and analyze a more general Laplacian for multivector-valued functions on manifolds. This allows us to formulate a higher codimensional analog of Jacobi's field equation.
\end{abstract}

\textit{AMS 2010 MSC:} {58Axx, 53A10, 30G35, 53Cxx}


\section{Introduction}
\textit{``The study of $2$-dimensional minimal surfaces in Euclidean $(2 + p)$-space is greatly simplified by the fact, that the Gauss map is antiholomorphic. [...] a weaker property holds for higer dimensions; namely that the Gauss map of a minimal submanifold of arbitrary codimension is harmonic"}, cf. \cite{RV}.\\
In this paper we show that already the first result can be generalized to arbitrary dimensions and arbitrary codimensions, cf. corollary \ref{monogenicGauss}. For that purpose the view on the generalized Gauss map should be modified in such a way that geometric calculus can be readily applied to it, cf. section \ref{sec1Lew}.\\
First considerations go back to the monograph \cite{HS}. However, a strict connection to the classical theory is missing there; instead, the authors develop `new' geometrical concepts. For example, their `spur' is the mean curvature vector field, cf. our discussion in section \ref{sec3Lew}.\\
So, in this paper we catch up on connecting the classical differential geometry and geometric calculus explicitly.\\
Moreover, the second vector derivative operator is precisely examined with an emphasis on its grade preserving part. This operator is named the \textit{graded Laplace operator} and differs from other known Laplacians, cf. section \ref{sec2Lew}. Especially, the graded Laplacian of the Gauss map can be identified with the curl of the mean curvature vector field, cf. theorem \ref{bladeI}. This result reflects the theorem from \cite{RV}, but it will be
deduced from geometric calculus.\\
Furthermore, preliminary considerations for Bernstein type theorems, as in \cite{Fis} and \cite{Wang}, are presented in section \ref{sec3Lew}.\\
On the other hand, the fact that the Gauss map of a minimal submanifold has to be harmonic also plays an important role in Bernstein type theorems since they reduce to Liouville type theorems for harmonic maps, cf. \cite{HJW}, \cite{JX}, \cite{JXY}. From the point of view of geometric calculus, even a stronger condition is satisfied. Namely that minimal submanifolds of arbitrary dimension and arbitrary codimension can be characterized by monogenic Gauss map, cf. corollary \ref{monogenicGauss}.\\
However, an intrinsic point of view (including Clifford bundles associated with tangent bundles, covariant derivatives, Dirac operators \ldots cf. \cite{LM}) would not lead to the desired results, see  discussions in subsection \ref{LewIntrinsicGeo}.
\medskip

Firstly, we summarize our

\section{Used Notations from Geometric Algebra}
As usual we denote by $\bigwedge_*\mathds{R}^{N}$ the \textit{exterior algebra} over $\mathds{R}^{N}$ with the \textit{exterior product} $\wedge$. The following automorphisms of $\bigwedge_*\mathds{R}^{N}$ play an important role in multilinear algebra:
\begin{itemize}
 \item the \textit{reversion} (denoted by \ $(.)^\mathbf{t}$\ ),
 \item the \textit{grade involution} (denoted by \ $\widehat{(.)}$\ ).
\end{itemize}
For an $r$-vector $A_r\in\bigwedge_r\mathds{R}^{N}$, with $r\in\{0,\ldots,N\}$, yields
\begin{equation}
 {A_r}^\mathbf{t}=(-1)^\frac{r(r-1)}{2}A_r \qquad \text{and}\qquad \widehat{A_r}=(-1)^rA_r.
\end{equation}

In an extension of the usual Euclidean scalar product for vectors (denoted by\ $\skalarProd{}{}$\ ) one defines a \textit{scalar product} for multivectors (denoted by\ $*$\ ). This generates
\begin{itemize}
 \item an \textit{inner product} (denoted by\ $\Stufe{.,.}{}$\ ) via
 \begin{equation}
  \Stufe{A,B}{}\coloneqq A*B^\mathbf{t},
 \end{equation}
 \item a \textit{left} and a \textit{right contraction} (denoted by\ $\aus$\ and\ $\ausR$\ resp.) via
 \begin{subequations}
 \begin{align}
 (X\wedge A)*B&\eqqcolon X*(A\aus B) \quad \text{for all $\textstyle X\in\bigwedge_*\mathds{R}^{N}$}
 \shortintertext{and}
 A*(B\wedge X)  &\eqqcolon (A\ausR B)*X\quad \text{for all $\textstyle X\in\bigwedge_*\mathds{R}^{N}$ resp.,}
 \end{align}
 \end{subequations}
 for multivectors $A,B\in\bigwedge_*\mathds{R}^{N}$.
\end{itemize}
The inner product induces a (squared) \textit{norm} on $\bigwedge_*\mathds{R}^{N}$:
\begin{equation}
 |A|^2\coloneqq \Stufe{A,A}{}.
\end{equation}

One disadvantage of the above products is that they are not invertible. This can be partially overcome by an introduction of the \textit{geometric product}, also known as \textit{Clifford product}. The geometric product will be denoted by juxtaposition and is a more fundamental multiplication. In fact, the previous products can be deduced from the geometric product, cf. \cite{Dorst}.

Even further products can be obtained, e.g. the \textit{commutator product} (denoted by\ $\times$\ ):\vspace{-1ex}
\begin{equation}
A\times B\coloneqq \frac{1}{2}(AB-BA).
\end{equation}

Moreover, the geometric product of a vector $a\in\mathds{R}^{N}$ with a multivector $B\in\bigwedge_*\mathds{R}^{N}$ decomposes into
\begin{subequations}
\begin{equation}
 aB=a\aus B + a \wedge B,
\end{equation}
which in the case of vectors $a,b\in\mathds{R}^{N}$ becomes
\begin{equation}
 ab=\skalarProd{a}{b}+a\wedge b.
\end{equation}
\end{subequations}
The geometric product of a bivector $\text{\boldmath $B$}\in\bigwedge_2\mathds{R}^{N}$ with $A\in\bigwedge_*\mathds{R}^{N}$ takes the form\vspace{-1ex}
\begin{equation}
 \text{\boldmath $B$}A=\text{\boldmath $B$}\aus A + \text{\boldmath $B$}\times A + \text{\boldmath $B$}\wedge A.
\end{equation}

\begin{rem}For further reading and elaborated clarifications of this subject we refer the reader to \cite[ch 1]{Fed}, \cite[ch 1]{HS}, \cite{Dorst}, \cite[ch 1]{Brackx}, \cite{DFM}.
In fact, the above interior multiplications have not been investigated in \cite{HS}. Indeed, a different `inner product' was introduced by Hestenes, cf. \cite[(1-1.21)]{HS}. Let us denote the latter by $\bullet_{\text{H}}$. It is actually not positive definite since $$ (e_1\wedge e_2) \bullet_{\text{H}}(e_1\wedge e_2) =-1$$ and forces to include grade-based exceptions in practice,  but these problems can be fixed by using contractions, cf. discussion in \cite{Dorst}.
\end{rem}


\section{The Unit Blades Manifold} \label{sec1Lew}

Let $\mathrm{i}=e_1\wedge\ldots\wedge e_{N}$ denote the unit pseudoscalar in $\bigwedge_*\mathds{R}^{N}$. The \textit{Hodge dual} of a multivector $A\in\bigwedge_*\mathds{R}^{N}$ is given by
\begin{equation}\label{HodgeDual}
 \star A \coloneqq {A}^{\mathbf{t}}\mathrm{i} = A^{\mathbf{t}}\aus \mathrm{i}.
\end{equation}

As usual we associate with a non-zero simple $m$-vector $a_1\wedge\ldots\wedge a_m$ the $m$-dimensional space $\mathcal{V}_m$ with frame $\{a_j\}_{j=1,\ldots,m}$. Hence,  $ v\in\mathcal{V}_m $ if and only if \ $ a_1\wedge\ldots\wedge a_m\wedge v = 0$.

The Hodge dual $\star(a_1\wedge\ldots\wedge a_m)$ represents the $k$-dimensional space, which is orthogonal to $\mathcal{V}_m$, where $k=N-m$.

Moreover, the orientation of $a_1\wedge\ldots\wedge a_m$ is transferred under the above association to an orientation of the subspace $\mathcal{V}_m$ so that $\mathcal{V}_m$ is a point in the \textit{oriented} Grassmann manifold \ $\widetilde{{\mathbf G}_m}(\mathds{R}^{N})$. \ Here, two $m$-blades (i.e. simple $m$-vectors) $A_m$ and $B_m$ are mapped onto the same (oriented) subspace if and only if ~ $A_m=\gamma B_m$ ~ for some \textit{positive} number $\gamma$.

Conversely, an (oriented) $m$-space can be mapped onto an $m$-vector via
\begin{equation}
 \mathcal{V}_m \text{ with frame } \{a_j\}_{j=1,\ldots,m} \longmapsto  a_1\wedge\ldots\wedge a_m.
\end{equation}

Indeed, a different choice of (ordered) basis shall give a different exterior product, but the two blades differ only by a (positive) scalar, namely the determinant of the basis change matrix.

Let $\big[\bigwedge_m \mathds{R}^{N} \big]$ denote the set of unit $m$-blades. Then the above mappings yield a bijection between unit $m$-blades and oriented $m$-subspaces.\\[-1ex]

Recall that the geometric algebra of $\mathds{R}^{N}$ is a metric space. So we can equip the subset of unit $m$-blades with the subspace topology. Using similar arguments as in \cite{Leicht} it follows that $\big[\bigwedge_m \mathds{R}^{N} \big]$ is locally homeomorphic to $\mathds{R}^{m\cdot k}$, where $k=N-m$, and can be endowed with a differentiable atlas. All in all,   $\big[\bigwedge_m \mathds{R}^{N} \big]$ is a differentiable manifold and will be called the \textbf{unit $m$-blades manifold}.

Furthermore, the Lie group $SO(N)$ acts transitively on $\big[\bigwedge_m \mathds{R}^{N} \big]$. For every point the isotropy group is $SO(m)\times SO(k)$ so that the unit $m$-blades manifold is diffeomorphic to the homogeneous space
\[
 \QR{SO(N)}{SO(m)\times SO(k)},
\]
the Grassmannian $\widetilde{{\mathbf G}_m}(\mathds{R}^{N})$. The advantage of considering the unit blades manifold instead of the homogeneous space is that we can operate with its points, namely blades, in the manner familiar to us from geometric calculus.\\[-1ex]

Grassmannians naturally appear in differential geometry:\\
In the following we will denote by $\Mfg$ an oriented $m$-dimensional smooth submanifold of the Euclidean vector space $\mathds{R}^{N}$. For any point $x\in\Mfg$, a parallel translation of the (oriented) tangent space $T_x\Mfg$ to the origin yields an (oriented) $m$-subspace of $\mathds{R}^{N}$ and, therefore, a point $\mathbf{g}(x)$ in the Grassmann manifold $\widetilde{{\mathbf G}_m}(\mathds{R}^{N})$.
\begin{defn}
 The mapping $x\mapsto \mathbf{g}(x)$ is called the \textbf{generalized Gauss map}.
\end{defn}

We have already seen how a unit $m$-blade $\I(x)$ can be assigned to $ \mathbf{g}(x)$. Hestenes calls the $m$-vector field $\I(.)$ the \textit{pseudoscalar of the manifold} $\Mfg$ as $\I(x)$ is the unit pseudoscalar in the geometric algebra of the tangent space $T_x\Mfg$. In geometric measure theory such a function is called an \textit{$m$-vector field orienting} $\Mfg$  or rather the \textit{orientation} of $\Mfg$, cf. \cite[4.1.7, 4.1.31]{Fed}, \cite[p. 132]{Simon}. To be more precise:\\ If $\tau_j=\tau_j(x)$, $j=1, \ldots, m$, is an orthonormal frame of $T_x\Mfg$, then
\begin{equation}
 \I(x)=\tau_1(x)\wedge\ldots\wedge \tau_m(x).
\end{equation}
\begin{defn}
 We call the mapping $x\mapsto\I(x)$ the \textbf{Gauss map}.
\end{defn}
 Summing up, the following diagram commutes:
\begin{equation}
\begin{tikzpicture}[>=stealth,anchor=base,baseline]
\node (a) at (180:3.14) {$\Mfg$};
\node (b) at (60:1.2) {$\widetilde{{\mathbf G}_m}(\mathds{R}^{N})$};
\node (c) at (-60:1.2)  {$\left[\bigwedge_m \mathds{R}^{N}\right]$};
\draw
(a) edge [->,out=40,in=180] node[above,sloped] {$\mathbf{g}$} (b)
    edge [->,out=-40,in=180] node[above,sloped] {$\I$} (c)
(b) edge [<->] node[right] {$\cong$} (c);
\end{tikzpicture}
\end{equation}
\begin{rem}
In what follows, the vectors $\tau_j$ will denote an orthonormal frame of the corresponding tangent space.
\end{rem}

\section{A Short Course on Geometric Calculus}\label{sec2Lew}
We start with careful investigations of multivector-valued functions on $\Mfg$, including new facts and terms, but also some basic results and definitions from \cite{HS} needed to make the paper self-contained since \textit{``the derived products in \cite{HS} have not quite been chosen properly"}, cf. the discussion in \cite{Dorst} where the problems were fixed.\vspace{-1ex}

\subsection{The First Derivative}
Let $a=a(x)$ be a tangent vector, i.e. $a\wedge\I=0$. The \textbf{directional derivative} of a function $F:\Mfg\to\bigwedge_*\mathds{R}^{N}$ in the direction $a$ is given by the usual limit:
\begin{equation}
 (\skalarProd{a}{\partial})F=\left(\skalarProd{a(x)}{\partial_x}\right)F(x)\coloneqq \lim_{t\to 0}\frac{F(x(t))-F(x(0))}{t}
\end{equation}
if it exists, where $x(t)\subset\Mfg$ with $x(0)=x$ and $x'(0)=a$.
\smallskip

\paragraph{General agreement:} All directional derivatives shall exist in the following.
\smallskip

\noindent Hestenes defines the \textbf{vector derivative} of a function $F:\Mfg\to\bigwedge_*\mathds{R}^{N}$  by
\begin{equation}
 \partial F\coloneqq \sum_{j=1}^m\tau_j\underset{\mathclap{\substack{\\ \\ \uparrow\\ \text{geometric} \\ \text{product} }}}{\ }(\skalarProd{\tau_j}{\partial})F.
 \end{equation}
For notational simplicity, the dependence on $x$ is here and will be frequently suppressed. The derivative operator is
\begin{equation}
\partial = \partial_x =\sum_{j=1}^m \tau_j(\skalarProd{\tau_j}{\partial}) =\sum_{j=1}^m \tau_j(x)(\skalarProd{\tau_j(x)}{\partial_x}).
 \end{equation}
The vector derivative does not depend on the choice of basis vectors and it can also be introduced in a coordinate-free manner, cf. \cite[sec. 7-2]{HS} \footnote{In a similar manner Pompeiu introduced his \textit{areolar derivative} already in 1912, cf. discussion in \cite[sec. 7.2.4]{GHT}.}. Furthermore, $\partial$ has the algebraic properties of a vector field, cf. \cite[sec. 4-1]{HS}:
\begin{equation}
 \P(\partial) = \partial \qquad \text{or}\qquad \I\wedge\partial = 0
 \end{equation}
where $\P=\P_\I$ denotes the projection onto $\I$. More precisely, it reads as $\P(\dot{\partial})=\dot{\partial}$ or $\I\wedge\dot{\partial}=0$ since $\partial$ does not operate on the projection or the pseudoscalar here.

As a matter of fact, it would be better to call $\partial F$ the \textit{left} vector derivative because the $\tau_j$'s are multiplied from the left side. Since the geometric product is not commutative, we should also consider the \textit{right} vector derivative\vspace{-1ex}
\begin{equation}
  \dot{F}\dot{\partial} \coloneqq \sum_{j=1}^m \Big((\skalarProd{\tau_j}{\partial})F\Big)\underset{\mathclap{\substack{\\[0.5ex] \uparrow\\ \text{geometric} \\ \text{product} }}}{\ }\tau_j.
 \end{equation}

\begin{rem}
 The major difference between the common Dirac operators and the vector derivative is that the latter also contains normal components. This is important for the considerations in section \ref{sec3Lew}.
\end{rem}

Seeing that $\partial$ behaves like a vector, the geometric product in $\partial F$ can be decomposed into a left contraction and a wedge product, cf. \cite[(2-1.21a)]{HS}:
\begin{equation}
 \partial F= \underset{\text{divergence}}{\partial\aus F}+ \ \ \underset{\text{curl}}{\partial\wedge F}
\end{equation}
\noindent where the above terms are chosen in accord with the standard nomenclature. We will not focus on the equivalent terms coming from $F\partial$, cf. rem. \ref{LewLeftRight}.\\

Functions with vanishing derivative operators are particularly interesting. We start with
\begin{defn}
Let $F:\Mfg \to \bigwedge_* \mathds{R}^{N}$.
\begin{itemize}
 \item[-] $F$ is called \textbf{divergence-free} iff $\partial\aus F\equiv 0$ on $\Mfg$.
 \item[-] $F$ is called \textbf{curl-free} iff $\partial\wedge F \equiv 0$ on $\Mfg$.
\end{itemize}
\end{defn}

\begin{rem}\label{LewLeftRight}
It is sufficient to consider the left operations here because transformations including the \textit{grade involution} yield:
 \begin{equation}
  \partial \aus F = \widehat{F \ausR \partial} \qquad  \text{and} \qquad  \partial \wedge F = -\widehat{F \wedge \partial}.
 \end{equation}
\end{rem}

\begin{ex}
The identity map $F(x)=x$ is curl-free, and for its divergence we obtain: \ $\skalarProd{\partial}{x}=m$.
\end{ex}

From \cite{Brackx} we adopt the
\begin{defn}
 Let $F:\Mfg \to \bigwedge_* \mathds{R}^{N}$.
\begin{itemize}
 \item[-] $F$ is called \textbf{left monogenic} iff $\partial F\equiv 0$ on $\Mfg$.
 \item[-] $F$ is called \textbf{right monogenic} iff $ F\partial \equiv 0$ on $\Mfg$.
 \item[-] $F$ is called \textbf{monogenic} iff $\partial F\equiv 0$  and $ F\partial \equiv 0$ on $\Mfg$.
\end{itemize}
\end{defn}

\begin{rem}\label{equivMonog}
For an $r$-vector-valued function $\text{\boldmath $F$}_r:\Mfg \to\bigwedge_r\mathds{R}^{N}$, ~ with  $r\in\{0,\ldots,N\}$, the both monogenicity terms are equivalent, but in general, left monogenicity does not imply right monogenicity, cf. the following
\end{rem}

\begin{rem}
 Left monogenic functions generalize the concept of \textit{holomorphic} functions to arbitrary dimensions and arbitrary manifolds.  Similarly, \textit{anti-holomorphic} functions are generalized by right monogenic ones.
\end{rem}

Transformations including the grade involution prove

\begin{prop}\label{PropMonogenicLew}
 $F$ is divergence- and curl-free if and only if $F$ is monogenic.
\end{prop}

\begin{rem}
Polynomials are in general not monogenic, cf. $\partial x = m$.
\end{rem}

\subsection{The Second Derivative}
The second derivative of $F$ decomposes in the following ways:
\begin{equation}
\begin{split}
 \partial^2 F &= \partial (\partial F) =\partial\aus(\partial\aus F) +
\underbrace{\partial\aus(\partial\wedge F ) +
\partial \wedge (\partial \aus F)}_{\text{same grade(s) as $F$}} +  \partial \wedge
(\partial \wedge F)  \\
&= (\partial\partial)F = (\partial \wedge \partial)\aus F +
\overbrace{(\skalarProd{\partial}{\partial}) F +
(\partial \wedge \partial)\times F } + (\partial \wedge \partial) \wedge F .
\end{split}
\end{equation}
Although ~ $\partial$ ~ is regarded as a vector, ~ $\partial\wedge\partial$ ~ does not vanish in general, cf.\linebreak \cite[sec. 4-1]{HS} and it has algebraic properties of a bivector.\\ Recall that $(\skalarProd{\partial}{\partial})$ behaves like a scalar and that the commutator product with a bivector is grade preserving. Hence, \ $(\skalarProd{\partial}{\partial})(.) +(\partial \wedge \partial) \times(.)$ \ acts in a graded linear way, namely
\begin{equation}
 (\skalarProd{\partial}{\partial})\text{\boldmath $F$}_r+(\partial \wedge \partial) \times \text{\boldmath $F$}_r = \Stufe{(\skalarProd{\partial}{\partial})\text{\boldmath $F$}_r+(\partial \wedge \partial) \times \text{\boldmath $F$}_r}{r}.
\end{equation}
This fact justifies the new
\begin{defn}
Let $F:\Mfg \to \bigwedge_* \mathds{R}^{N}$. The \textbf{graded Laplace operator} of $F$ on $\Mfg$ is given by\vspace{-1ex}
 \begin{align}
   \diamondsuit F \coloneqq&\ \partial\aus(\partial\wedge  F) + \partial \wedge (\partial \aus F)\tag{4.10-i}\\
	=&\ (\skalarProd{\partial}{\partial}) F + (\partial \wedge \partial)\times F. \tag{4.10-ii}
 \end{align}
\end{defn}
\setcounter{equation}{10}
The symbol $\diamondsuit$ should remind that we are dealing with the grade preserving part of the second derivative:
\begin{equation}
 \Stufe{\partial^2 \text{\boldmath $F$}_r}{r}=\diamondsuit \text{\boldmath $F$}_r
\end{equation}
by pushing together $\Stufe{\ }{}$.
\begin{rem}
As in the case of the vector derivative, the presence of the normal components is relevant. Disregarding them, we would arrive at the Hodge-de Rham-Laplace operator, but then our theorem \ref{bladeI} will not be obtained (cf. subsection \ref{LewIntrinsicGeo}).
\end{rem}

Transformations including grade involution show that it suffices to look at the left operations here. Moreover, the graded Laplacian is given by the arithmetic mean of the second left and second right derivatives:
\begin{equation}\tag{4.10-iii}
 \diamondsuit F = \frac{1}{2}\left(\partial^2 F + F\partial^2\right).
\end{equation}

 Furthermore, the `scalar part' of $\diamondsuit$ reads in \textit{Riemannian normal coordinates} as\vspace{-1ex}
\[
 \skalarProd{\partial}{\partial} = \sum_{j=1}^m (\skalarProd{\tau_j}{\partial})\, (\skalarProd{\tau_j}{\partial}).
\]
This operator is well known in classical differential geometry as the \textbf{Laplace-Beltrami operator} $\Delta_\Mfg$ on $\Mfg$, cf. \cite[p. 163]{DHT2} or \cite[prop. 1.2.1]{Sms}.

It follows immediately that for a \textit{scalar} function $\varphi:\Mfg\to\mathds{R}$ the two Laplacians are the same:
\begin{equation}
  \diamondsuit \varphi = \Delta_\Mfg \varphi.
\end{equation}

Recall that a real function $\varphi$ is called \textit{harmonic} iff $\Delta_\Mfg \varphi\equiv 0$ on $\Mfg$. Hence, functions with vanishing graded Laplace operator are also interesting:
\begin{defn}
Let $F:\Mfg \to \bigwedge_* \mathds{R}^{N}$. We call $F$ to be \textbf{graded-harmonic} on $\Mfg$ iff $\diamondsuit F \equiv 0$ on $\Mfg$.
\end{defn}

Proposition \ref{PropMonogenicLew} and the above expressions of the graded Laplacian verify

\begin{prop}
 Monogenic functions are graded-harmonic.
\end{prop}

\begin{rem}
 The converse is not true in general: \hfill $ \partial x = m \hfill \text{and} \hfill \diamondsuit x = 0. $
\end{rem}

\begin{rem}
 Only left monogenicity or only right monogenicity is, in general, not sufficient for the vanishing of the graded Laplacian. However, it is sufficient in the case of $r$-vector-valued functions, cf. rem. \ref{equivMonog}.
\end{rem}

\begin{rem}\label{notHarm}
 In general, harmonic and graded-harmonic multivector-valued functions are distinct. Indeed, the target space of a multivector-valued function is the space $\bigwedge_* \mathds{R}^{N}$. Since this space is flat, its Christoffel symbols vanish. Consequently, a harmonic function $F:\Mfg\to\bigwedge_* \mathds{R}^{N}$, in the sense of Eells and Sampson \cite{ES}, will be a solution of $\Delta_\Mfg F \equiv 0$, the Euler-Lagrange equation of a certain energy.

 However, we consider a more general operator. We have
 \begin{equation}
  \diamondsuit F = \Delta_\Mfg F + (\partial\wedge\partial)\times F \tag{4.10-iv}
 \end{equation}
 hence the solutions of $\diamondsuit F \equiv 0$ deserve our special attention.
\end{rem}

\subsection{The Shape Operator}\label{shape}
In \cite[sec. 4-2]{HS} Hestenes and Sobczyk introduce and examine the \textit{shape} of a function $F:\Mfg \to\bigwedge_* \mathds{R}^{N}$. The \textbf{shape operator} of $F$ on $\Mfg$ is given, cf. \cite[(4-2.14)]{HS}, by
\[
\S(F) \coloneqq \dot{\partial}\dot{\P}(F)
\]
where the projection operator is being differentiated. Setting $\S_a\coloneqq\dot{\partial}\wedge\dot{\P}(a)$, the shape of a vector-valued function $a=a(x)$ can be expressed by
\[
\S(a) = \S_a + \skalarProd{\H}{a}
\]
with a \textit{normal} vector field $\H=\H(x)$, cf. \cite[(4-2.18)]{HS}. For `the spur' $\H$, we use a notation different from \cite{HS} since it turned out to be a well-known invariant in differential geometry:

Let $a$ and $b$ be vector fields on $\Mfg$, then $a\aus\S_b=b\aus\S_a$ and the normal part of the directional derivative of $b$ in the direction $a$ can be expressed by
\[
 \big( (\skalarProd{a}{\partial})b\big)^\perp = a\aus \S_b,
\]
cf. \cite[(4-3.16)]{HS}. All in all, ~ $(.)\aus\S(.)$ ~ is a symmetric bilinear form on $T_x\Mfg$ with values in
the normal space $N_x\Mfg$, namely the \textbf{second fundamental form} of $\Mfg$ at $x$, cf. \cite[prop. 2.2.2]{Sms}.\\[-1ex]

Writing out the expression \cite[(4-2.20)]{HS} in coordinates we arrive at
 \begin{equation}\label{mcv1}
  \H = \partial_a\aus \S_a = \textstyle\sum\limits_{j=1}^m \tau_j \aus \S_{\tau_j},
 \end{equation}
hence $\H=\H(x)$ is the trace of the second fundamental form at $x$ and, by definition, the \textbf{mean curvature vector} at $x$ on $\Mfg$, cf. \cite[p. 68]{Sms}, \cite[p. 301]{DHT1} or \cite[p. 158]{DHT2}. Note that we have introduced it without the factor $\frac{1}{m}$, and the `harpoon' emphasizes that we are dealing with a vector.\bigskip

Straightforward algebraic calculations show the important relation of the shape to the second derivative operator, cf. \cite[sec. 4-3]{HS}:
\[
 \partial \wedge \partial = \S(\partial).
\]
Finally, we arrive at an additional expression for the graded Laplacian:
\begin{equation}
 \diamondsuit F = \Delta_\Mfg F + \S(\dot{\partial})\times\dot{F}. \tag{4.10-v}
\end{equation}

Furthermore, the graded Laplacian differs from the second derivative. For a \textit{scalar} function $\varphi:\Mfg\to\mathds{R}$, we already know:
\begin{equation}
  \partial^2\varphi = \diamondsuit \varphi + \partial\wedge\partial \varphi = \diamondsuit \varphi + \S(\partial\varphi).
\end{equation}

\begin{ex}
 Let us have a short look at the identity map $F(x)=x$ on $\Mfg$:\\
 Since $\skalarProd{\partial}{x}=m$ and $\partial\wedge x = 0$, we obtain
 \[
  \partial x = m, \quad \text{and besides}\quad \partial^2x=0 \ \text{~ and ~} \ \diamondsuit x = 0.
 \]
 Moreover,
 \begin{equation}
\begin{split}
   (\partial \wedge\partial)x &= \S(\dot{\partial})\dot{x} = (\skalarProd{\partial_a}{\dot{\partial}_x})(\S_a\ausR \dot{x} + \S_a \wedge \dot{x})=\\
 &= -\partial_a\aus\S_a  +   \partial_a\wedge\S_a= -\H
\end{split}
 \end{equation}
as $\partial_a\wedge\S_a\equiv0$, cf. \cite[(4-2.23)]{HS}. Therefore,
\begin{equation}
 \Delta_\Mfg x =(\skalarProd{\partial}{\partial})x = \partial^2 x -(\partial\wedge \partial)x = \H.
\end{equation}
Admittedly, we have derived a classical result, cf. \cite[thm. 2.1]{Oss} or \cite[p. 305]{DHT1}:

Denoting the standard basis of $\mathds{R}^{N}$ by $\{e_1, \ldots, e_{N}\}$, we obtain
\[
 \diamondsuit\Stufe{x, e_i}{} = \Delta_\Mfg \Stufe{x, e_i}{}= (\skalarProd{\partial}{\partial}) \Stufe{x, e_i}{}
= \Stufe{(\skalarProd{\partial}{\partial}) x, e_i}{}= \Stufe{\H, e_i}{}.
\]
This example shows that the graded Laplacian does \textit{not} act coordinate-wise in contrast to the Laplace-Beltrami-operator (cf. also rem. \ref{notHarm}):
\[
0 = \diamondsuit x = \sum_{i=1}^{N} \diamondsuit \left( \Stufe{x, e_i}{} e_i\right) \neq \underbrace{\sum_{i=1}^{N} \left( \diamondsuit\Stufe{x, e_i}{} \right) e_i}_{=~(\skalarProd{\partial\;\!}{\!\;\partial})x~ =~
\Delta_\Mfg x} = \H.
\]
Algebraically, it is clear that $\Delta_\Mfg =(\skalarProd{\partial}{\partial}) $ acts coordinate-wise since it behaves like a scalar factor. To sum up, we have:
\[ \diamondsuit x = 0 \qquad \text{while} \qquad \Delta_\Mfg x = \H.\]
\end{ex}

\paragraph{Conclusion}
We have introduced the new graded Laplace operator $\diamondsuit$ as grade preserving part of the second derivative operator $\partial^2$ on $\Mfg$. For scalar-valued functions, $\diamondsuit$ coincides with the Laplace-Beltrami operator $\Delta_\Mfg$, but in general, the graded Laplacian contains much more information. We will explicitly see the difference between these operators in \ref{Norm2FF} and \ref{bladeNormlvf} while computing the graded Laplacian of the Gauss map.

\section{Geometric Calculus of the Gauss Map}\label{sec3Lew}
In \cite{HS} Hestenes and Sobczyk examined the Gauss map $\I$, called by them the pseudoscalar of the manifold, intensively and gave many different expressions for its derivatives. However, a clear geometric interpretation is missing in their monograph. This may be explained by the fact that they were not aware of the close relation to differential geometry and that was especially expressed by the missing discovery of the mean curvature vector in their relations. Indeed, they rediscovered it as `the spur' and realized its geometric importance \footnote{For example, it appears in the divergence theorem, cf. \cite[sec. 7-3]{HS} and \cite[sec. 3-5]{DHT2}.} :\\
\textit{``The spur}
\begin{itemize}
 \item[-] \textit{is everywhere orthogonal to tangent vectors of the manifold, which fits nicely with the connotation of `spur' as something that `sticks out' (of the manifold)"}, cf.  \cite[p. 151]{HS}
 \item[-] \textit{tells us some fundamental things about the manifold"}, cf. \cite[p. 164]{HS}
 \item[-] \textit{is the appropriate generalization of mean curvature to arbitrary vector manifolds"}, cf. \cite[p. 197]{HS}
\end{itemize}
In a later survey Hestenes mentioned:

\textit{``As far as I know, the spur was not identified as significant geometrical concept until it was first formulated in G[eometric] C[alculus]."} cf. \cite{Hes}
\medskip

\noindent Recognizing the mean curvature \textit{vector}, we will now catch up the differential-geometric meaning of the algebraic equations from \cite[sec. 4-4]{HS}.

\subsection{The First Derivative of the Gauss Map}
In \cite[(4-4.7)]{HS} Hestenes obtained
\begin{thm}\label{diffGaussmap}
 \quad $ ~ \partial\I = - \H\I$.
\end{thm}
Recall the
\begin{defn}
 $\Mfg\subset\mathds{R}^{N}$ is called \textbf{minimal submanifold} iff $\H\equiv0$ on $\Mfg$.
\end{defn}
So, an interpretation of theorem \ref{diffGaussmap} yields the beautiful new
\begin{cor}\label{monogenicGauss}
 $\Mfg\subset\mathds{R}^{N}$ is a minimal submanifold if and only if its Gauss map is monogenic.
\end{cor}
\begin{rem}
It is known that $2$-dimensional minimal surfaces can be characterized by (anti-)holomorphic generalized Gauss map, cf. \cite{Chern}, \cite{HO}, \cite{Fuji}. Recall that left monogenicity generalizes holomorphy and right monogenicity antiholomorphy. Since the Gauss map is an $m$-vector field, the monogenicity terms are equivalent and the above corollary is an adequate statement in all dimensions and all codimensions.
\end{rem}
\begin{rem}
Theorem \ref{diffGaussmap} leads to a coordinate-free expression of $\H$, cf. \cite[(4-4.6)]{HS}, which -- rewritten in local coordinates -- yields the classical expression for the mean curvature vector field, cf. \eqref{mcv1} above and \cite[p. 68]{Sms}.
\end{rem}

\subsection{Divergence and Curl of the Mean Curvature Vector Field}\label{DivRotH}
\begin{lem}\label{divMKV} Let $X=X^{\perp}$ be a normal vector field on $\Mfg$. Then
\begin{equation*}
 \skalarProd{\partial}{X}+\skalarProd{\H}{X}\equiv0.
\end{equation*}
\end{lem}
This is a classical relation, cf. \cite[p. 199]{DHT2}, but we give a
\begin{proof}[Coordinate-free proof]
 Taking curl on both sides of $X\aus\I\equiv0$, we gain:
 \[
   \dot{\partial}\wedge(\dot{X}\aus\I) + \dot{\partial}\wedge(X\aus\dot{\I}) \equiv 0.
 \]
Since $X$ is normal and $\dot{\partial}$ behaves like a tangent vector, the first summand can be rewritten as follows:
\begin{align*}
 \dot{\partial}\wedge(\dot{X}\aus\I) &= (\skalarProd{\partial}{X})\I, \\
 \intertext{and the second summand as follows:}
 \dot{\partial}\wedge(X\aus\dot{\I}) &= - X\aus(\partial\wedge\I)  = X\aus(\H\wedge \I) = (\skalarProd{X}{\H})\I.  
\end{align*}
\end{proof}

\noindent Recall that the mean curvature vector field $\H$ is called to be \textbf{parallel} on $\Mfg$ iff
\begin{equation}
  \big((\skalarProd{\tau_j}{\partial})\H\big)^\perp \equiv 0 \quad \text{for all ~} j=1, \ldots, m.
\end{equation}
Straightforward computations show the new
\begin{prop}\label{propHparallel}
 $\H$ is parallel if and only if ~ $\partial\wedge\H\equiv0$.
\end{prop}
\noindent The advantage of the latter expression is that we are no longer forced to consider coordinates or normal parts and even have a nice expression for the parallelism of $\H$ in the language of geometric calculus.
Subsequently, we will talk about curl-free mean curvature vector field.

\subsection{The Second Derivative of the Gauss Map}
With regard to lemma \ref{divMKV} the expression \cite[(4-4.10)]{HS} can be simplified to
\begin{equation}
 \partial^2\I = -(\partial \wedge \H)\I.
\end{equation}
In fact, only the graded Laplacian remains from the second derivative of $\I$, cf. the grade comparison arguments \cite[(4-4.8), (4-4.11)]{HS}. Therefore, we arrive at the new
\begin{thm}\label{bladeI}
   \quad $\diamondsuit\I = - (\partial\wedge\H)\I$.
\end{thm}
And, a reinterpretation yields
\begin{cor}\label{interestingGauss}
  $\Mfg\subset\mathds{R}^{N}$ has a curl-free mean curvature vector field if and only if its Gauss map is graded-harmonic.
\end{cor}
\begin{rem}
In view of proposition \ref{propHparallel}, theorem \ref{bladeI} is a geometric calculus analog of Ruh and Vilms theorem, which links the tension field of the generalized Gauss map to the covariant derivative (in the normal bundle) of the mean curvature vector field as cross-sections, cf. \cite{RV}. Hence, the Gauss map $\I$ is graded-harmonic if and only if the generalized Gauss map $\mathbf{g}$ is harmonic in the sense of \cite{ES}. Indeed, rewriting $\diamondsuit \I \equiv 0$ in local coordinates yields an elliptic partial differential equation of second order.
\end{rem}

\subsection{The Normal Space}
The Gauss map is named after C. F. Gauss, who introduced it as \textit{normal map} for surfaces in $\mathds{R}^3$. For (oriented) hypersurfaces in $\mathds{R}^{m+1}$, one can define the Gauss map as a map into the unit sphere $\mathds{S}^m$ by considering the unit normal vector field. In higher codimensions we have more than one normal direction, and we are forced to deal with Grassmannians. We have considered the tangent unit simple $m$-vector field $\I$ so far. Certainly, we can also concentrate on the normal space and regard $x\mapsto\EuScript{N}(x)$ as the Gauss map where $\EuScript{N}=\EuScript{N}_k(x)$, with $k=N-m$, is the unit pseudoscalar in the geometric algebra of the normal space $N_x\Mfg$. The above calculations and corollaries will remain true, which is consistent with the fact that the Grassmannians $\widetilde{{\mathbf G}_m}(\mathds{R}^{N})$ and $\widetilde{{\mathbf G}_k}(\mathds{R}^{N})$ are diffeomorphic.
\medskip

\noindent Recall that at each point $x\in\Mfg$ we have the decomposition: \quad $\mathrm{i} = \I\, \EuScript{N}.$\\
In other words, $\EuScript{N}$ is the Hodge dual of $\I$, and with \eqref{HodgeDual} it follows:
\begin{equation}
 \EuScript{N} =\star \I = \I^\mathbf{t}\mathrm{i} = (-1)^{\frac{m(m-1)}{2}}\I\mathrm{i}.
\end{equation}
Hence, we gain similar formulae for the derivatives of $\EuScript{N}$:
\begin{equation}
 (\skalarProd{a}{\partial})\EuScript{N} = -\S_a\EuScript{N}, \quad \partial\EuScript{N}=-\H\EuScript{N},
 \quad \partial^2\EuScript{N}=-(\partial\wedge\H)\EuScript{N}.
\end{equation}
A grade comparison argument shows
\begin{align*}
  \partial\aus\EuScript{N}=-\H\aus\EuScript{N}, \qquad \partial\wedge\EuScript{N} = - \H\wedge\EuScript{N} \equiv 0
 \end{align*}
since \ $\H$ \ is normal and   $$ \partial\aus(\partial\aus\EuScript{N}) = -(\partial\wedge\H)\aus\EuScript{N} = - \dot{\partial}\aus(\dot{\H}\aus\EuScript{N}) \equiv 0 , $$
as \ $\dot{\H}\aus\EuScript{N}$ \ is a normal $(k-1)$-blade and \ $\dot{\partial}$ \ behaves like a tangent vector. Again, only the graded Laplacian is left over from the second derivative of $\EuScript{N}$, and we obtain
\begin{equation}\label{bladeN}
 \diamondsuit \EuScript{N} = -(\partial\wedge\H)\EuScript{N}.
\end{equation}
Consequently, the above theorems \ref{diffGaussmap} and \ref{bladeI} as well as corollaries \ref{monogenicGauss} and \ref{interestingGauss} remain true if we consider $\EuScript{N}$ instead of $\I$:
\begin{alignat*}{2}
\text{thm. \ref{diffGaussmap}:\qquad}&& \H&=-(\partial\I)\I^{-1}=-(\partial\EuScript{N})\EuScript{N}^{-1} \\
 \text{thm. \ref{bladeI}:\qquad}&& \partial\wedge\H&=-(\diamondsuit\I)\I^{-1}=-(\diamondsuit\EuScript{N})\EuScript{N}^{-1}
\end{alignat*}

and
\begin{longtable}{|lccc|}\hline
&&&\\[-0.25em]
 cor. \ref{monogenicGauss}:& $\H\equiv 0$  & $\Longleftrightarrow$ & $\partial \I \equiv 0\quad \Leftrightarrow\quad \partial \EuScript{N} \equiv 0$\\
 &$\Mfg$ minimal && Gauss-map monogenic\\[1em]\hline
 \multicolumn{1}{l}{}&&&\multicolumn{1}{c}{} \\
 \hline
 &&&\\[-0.25em]
  cor. \ref{interestingGauss}:& $\partial \wedge \H\equiv 0$ & $\Longleftrightarrow$ & $\diamondsuit \I \equiv 0\quad \Leftrightarrow\quad \diamondsuit \EuScript{N} \equiv 0$\\
 &$\H$ curl-free&& Gauss-map graded-harmonic\\[-0.5em]
 &&&\\\hline
\end{longtable}

\begin{rem}
 It is the geometric product, which allows us to work in a coordinate-free manner and to write down universal formulae. However, we have the different decompositions
 \begin{longtable}{lccccc}
 &$\partial \I = \partial \wedge \I$ & \hspace{2ex} and\hspace{2ex} & $\diamondsuit \I = \partial \aus (\partial \wedge \I)$ & since &  $\partial \aus \I \equiv 0$\\[2ex]
  whereas ~ &$\partial \EuScript{N} = \partial \aus \EuScript{N}$ &\hspace{2ex} and\hspace{2ex} & $\diamondsuit \EuScript{N} = \partial \wedge (\partial \aus \EuScript{N})$ & since & $\partial \wedge \EuScript{N} \equiv 0$.
\end{longtable}
\end{rem}

\begin{rem}
 With an orthonormal frame $\{n_1\ldots,n_k\}$ of $N_x\Mfg$, we gain the well-known expression for the mean curvature vector
 \begin{equation}
  \H = - (\partial \aus \EuScript{N})\EuScript{N}^{-1} = - \sum_{\alpha=1}^k \big(\skalarProd{\partial}{n_\alpha}\big)n_\alpha,\vspace{-1ex}
 \end{equation}
cf. \cite[sec. 4.3]{DHT1} resp. \cite[sec. 3.2]{DHT2}.
\end{rem}

\subsection{The Norm of the Second Fundamental Form}\label{Norm2FF}
The aim of this subsection is to show how the norm of the second fundamental form is contained in the graded Laplacian of the Gauss map. As before, $\{\tau_1, \ldots, \tau_m\}$ will denote an orthonormal frame of the tangent space $T_x\Mfg$ and $\{n_1, \ldots, n_k\}$ an orthonormal frame of the normal space $N_x\Mfg$. Although we always work at a particular point $x\in\Mfg$, we will again suppress the dependence on $x$ for notational simplicity.
As already seen in subsection \ref{shape}, the second fundamental form of $\Mfg$ at $x$ can be expressed by the shape operator and is given in the language of geometric calculus by \ $(.)\aus\S(.)$.\linebreak In particular, we have
\begin{subequations}
 \begin{equation}
 \tau_j\aus\S_{\tau_l} = \sum_{\alpha=1}^k(\skalarProd{n_\alpha}{\big((\skalarProd{\tau_l}{\partial})\tau_j\big)})n_\alpha = \sum_{\alpha=1}^k h_{\alpha jl}\, n_\alpha,
 \end{equation}
with the coefficients of the second fundamental form
\begin{equation}
 h_{\alpha jl}\coloneqq \skalarProd{n_\alpha}{\big((\skalarProd{\tau_l}{\partial})\tau_j\big)} = -\skalarProd{\tau_j}{\big((\skalarProd{\tau_l}{\partial})n_\alpha\big)}=(n_\alpha\wedge\tau_j)*\S_{\tau_l}.
\end{equation}
\end{subequations}
For the mean curvature vector we gain the familiar expression
\begin{equation}
 \H = \sum_{\alpha=1}^k\sum_{l=1}^m h_{\alpha ll}\, n_\alpha,
\end{equation}
and the norm of the second fundamental form is given by
\begin{equation}
  |\mathrm{B}|^2 \coloneqq \sum_{j,l=1}^m\sum_{\alpha=1}^k (h_{\alpha jl})^2.
\end{equation}
In geometric calculus this can be expressed in a coordinate-free manner via
\begin{lem}\label{NormIIFF}  \
$  |\mathrm{B}|^2 = \Stufe{\S(\partial_a), \S_a}{}.$
\end{lem}
\begin{proof}
As \ $\partial_a=\sum\limits_{j=1}^{m}\tau_j (\skalarProd{\tau_j}{\partial_a})$ \ we have
 \begin{align*}
  \Stufe{\S(\partial_a), \S_a}{} &= \sum_{j=1}^m \Stufe{\S(\tau_j), \S_{(\skalarProd{\tau_j\;\!}{\!\;\partial_a})a}}{} = \sum_{j=1}^m \left|\S_{\tau_j}\right|^2 =  \sum_{j=1}^m\left|\sum_{l=1}^m \tau_l\ (\tau_l\aus\S_{\tau_j})\right|^2 = \\ &= \sum_{j,l=1}^m (\tau_l\aus\S_{\tau_j})^2 = |\mathrm{B}|^2
 \end{align*}
since the $\tau_l$'s are orthonormal and $\tau_l\aus\S_{\tau_j}$ is a normal vector.
\end{proof}

Furthermore, we need the relation
\begin{lem}\label{JacShapeBivec}
 \ $(\S_b\times\S_a)\times\I=0$.
\end{lem}
\begin{proof}
Since the shape bivector anticommutes with the pseudoscalar, cf. \cite[sec. 4-4]{HS}, we have
\begin{align*}
 (\S_b\times\I)\times\S_a& = \frac{1}{2} (\S_b\I\S_a-\S_a\S_b\I) = \frac{1}{2}(-\S_b\S_a\I + \S_a\I\S_b)=\\
  &= (\S_a\times \I)\times\S_b.
\end{align*}
The Jacobi identity for the commutator product \cite[(1-1.56c)]{HS} yields then the desired result:
\[
 (\S_b\times\S_a)\times\I= (\I\times\S_a)\times\S_b - (\I\times\S_b)\times\S_a = 0.
\]
\end{proof}

We will now examine the graded Laplace operator of the Gauss map:
\begin{align*}
\diamondsuit \I &= \Stufe{\partial^2\I}{m} = \Delta_\Mfg\I + \Stufe{\S(\dot{\partial})\dot{\I}}{m}=\\
   &= \Delta_\Mfg \I - \Stufe{\S(\partial_a)\S_a\I}{m}=\\
   &= \Delta_\Mfg\I - \big(\S(\partial_a)*\S_a\big)\I - \big(\S(\partial_a)\times\S_a\big)\times\I - \Stufe{\big(\S(\partial_a)\wedge\S_a\big)\I}{m}
\end{align*}
where we used the fact that $\S_a$ is a bivector. By lemma \ref{NormIIFF} we have
\[
 \S(\partial_a)*\S_a = -  |\mathrm{B}|^2
\]
and by lemma \ref{JacShapeBivec} we obtain
\[
 \big(\S(\partial_a)\times\S_a\big)\times\I\equiv0.
\]
Hence, we gain the following formula for the graded Laplacian of the Gauss map:
\begin{equation}
  \boxed{\diamondsuit\I = \Delta_\Mfg\I + |\mathrm{B}|^2\I - \Stufe{\big(\S(\partial_a)\wedge\S_a\big)\I}{m}}~. \label{DiamondIausgesch}\vspace{1ex}
\end{equation}
Note that the $4$-vector $\S(\partial_a)\wedge\S_a$ does not generally vanish. Our theorem \ref{bladeI} then takes the form
\begin{prop}[Jacobi's field equation] \label{Jacobisfieldeq} The Gauss map on $\Mfg$ fulfills
 \[ \Delta_\Mfg\I + |\mathrm{B}|^2\I - \Stufe{\big(\S(\partial_a)\wedge\S_a\big)\I}{m}
 +(\partial\wedge\H)\I \equiv 0.\]
\end{prop}

\hfill

So, corollary \ref{interestingGauss} implies
\begin{lem}\label{rotH}
 Let $\Mfg$ have a curl-free mean curvature vector field. Then the Gauss map on $\Mfg$ fulfills
 $$ \Delta_\Mfg\I + |\mathrm{B}|^2\I - \Stufe{\big(\S(\partial_a)\wedge\S_a\big)\I}{m}\equiv0.$$
\end{lem}

We apply this result to prove
\begin{lem}\label{gewuenschteUngl}
  Let $\Mfg$ have a curl-free mean curvature vector field. If for some fixed $I\in\left[\bigwedge_m\mathds{R}^{N}\right]$ the Gauss map satisfies $\Stufe{\I(x),I}{}>0$ for all $x\in\Mfg$, then\vspace{-1ex}
  $$ -\Delta_\Mfg \ln \Stufe{\I,I}{}=  |\mathrm{B}|^2 + \frac{|\partial_a\Stufe{\S_a\I,I}{}|^2- \Stufe{\big(\S(\partial_a)\wedge\S_a\big)\I,I}{}\Stufe{\I,I}{}}{\Stufe{\I,I}{}^2}\raisebox{-1.5ex}{.} $$
\end{lem}
\begin{proof}
 We have\vspace{-1ex}
 \begin{align*}
  -\Delta_\Mfg \ln \Stufe{\I,I}{} &= \frac{|\partial\Stufe{\I,I}{}|^2-\Stufe{\I,I}{}\Delta_\Mfg\Stufe{\I,I}{}}{\Stufe{\I,I}{}^2}=\\[2ex]
  &=\frac{|\partial_a\Stufe{\S_a\I,I}{}|^2-\Stufe{\I,I}{}\Stufe{\Delta_\Mfg\I,I}{}}{\Stufe{\I,I}{}^2}\raisebox{-1.5ex}{,}
 \end{align*}
and the statement follows by lemma \ref{rotH}.
\end{proof}

\begin{rem}
A crucial ingredient in derivations of Bernstein type theorems is the superharmoniticity of \ $\ln \Stufe{\I,I}{}$ \ or rather the estimate
\[
 -\Delta_\Mfg \ln \Stufe{\I,I}{} \geq \delta |\mathrm{B}|^2 \qquad \text{for a \ $\delta >0$}
\]
on $\Mfg$, cf. \cite{Fis}, \cite{Wang}, \cite{JXY}. We will see that this is the case on hypersurfaces, cf. prop. \ref{LewProp1} below, but can also be obtained in the following situations:
 \begin{itemize}
  \item on minimal surfaces $\Mfg\subset\mathds{R}^{2+k}$ (two-dimensional),
  \item on minimal three-dimensional cones,
 \end{itemize}
cf. \cite{Fis}, \cite{Wang}. Due to the non-trivial example of a minimal non-parametric \linebreak $4$-dimensional cone in $\mathds{R}^{7}$ constructed by Lawson and Osserman in \cite{LO}, such an estimate is a priori not valid in dimensions $\geq4$ and codimensions $\geq3$. However, it can be obtained in all dimensions and all codimensions under suitable additional conditions, cf. \cite{Sms}, \cite{Reilly}, \cite{Fis}, \cite{HJW}, \cite{JX}, \cite{Wang}, \cite{JXY}.
\end{rem}

\begin{rem}
 Of course, we can again consider $\EuScript{N}$ instead of $\I$ in the last examinations. So, especially in the case of codimension $1$, the situation gets much simpler:
\end{rem}

\subsection{Hypersurfaces in Euclidean Space}\label{bladeNormlvf}
``\textit{For a hypersurface in Euclidean space, the role of its pseudoscalar $\I$ can be taken over by its normal $n$}'', cf. \cite[sec. 5-2]{HS}. The shape bivector becomes:
\[
 \S_a=n\wedge\big((\skalarProd{a}{\partial}) n\big),
\]
cf. \cite[(5-2.2)]{HS} and is, in fact, a \textit{simple} bivector.

For the second fundamental form we obtain the expression
\[
 b\aus\S_a = -(\skalarProd{b}{\big((\skalarProd{a}{\partial}) n\big)})\, n,
\]
in exact agreement with \cite[p. 299]{DHT1}.

Since there is the sole normal direction, the mean curvature vector differs from the unit normal only by a scalar factor, namely the \textbf{mean curvature}
\[
 \EuScript{H} = - \skalarProd{\partial}{n}.
\]

\begin{rem}
 The sign in the definition of the mean curvature depends on the choice of the direction of the normal vector. However, the definition of the mean curvature \textit{vector} is unique. Osserman \cite[p. 1095]{Oss} gives a picturesque description:\\
\textit{``\ $\H$ may be pictured as pointing toward the `inside' of $\Mfg$, in the sense that if $\Mfg$ is deformed by moving each point in the direction of the mean curvature vector at that point, then the volume of $\Mfg$ will initially decrease.''}
\end{rem}

Furthermore, since $\S_a$ is a simple bivector having $n$ as an exterior factor, we get the following  relation:
\begin{equation}\label{tolltollinEins}
  \S(\partial_a)\wedge\S_a \equiv 0.
\end{equation}
Thus, with \eqref{DiamondIausgesch} we obtain for the graded Laplacian of $n$ :
\begin{equation}\label{LewEQ1}
\diamondsuit n = \Delta_\Mfg n + |\mathrm{B}|^2 n.
\end{equation}
Moreover, we have
\[
 (\partial\wedge\H)\times n = \sum_{j=1}^m \tau_j \ \skalarProd{n}{\big((\skalarProd{\tau_j}{\partial})\H\big)}.
\]
Due to
\[
 (\skalarProd{\tau_j}{\partial})\H = (\skalarProd{\tau_j}{\partial}) \big(\EuScript{H} n\big)= n\,  (\skalarProd{\tau_j}{\partial})\EuScript{H} + \EuScript{H}\, \big((\skalarProd{\tau_j}{\partial}) n\big)
\]
and to
\[
 \skalarProd{n}{\big((\skalarProd{\tau_j}{\partial}) n\big)}=0 \qquad \text{ since \ \ $n^2=\skalarProd{n}{n}=1$},
\]
we conclude that
\begin{equation}\label{LewEQ2}
  (\partial\wedge\H)\times n = \sum_{j=1}^m\tau_j\, (\skalarProd{\tau_j}{\partial})\EuScript{H}=\partial \EuScript{H}.
\end{equation}

In view of \eqref{LewEQ1} and \eqref{LewEQ2}, the equation \eqref{bladeN} -- in the codimension one case -- becomes
\begin{equation}\tag{\ref{bladeN}'}
  \Delta_\Mfg n + |\mathrm{B}|^2 n + \partial\EuScript{H}=0,
\end{equation}
 i.e. exactly the \textit{Jacobi's field equation}, cf. \cite[p. 163]{DHT2}. Hence, its generalization for higher codimension was given by our theorem \ref{bladeI} or rather by prop. \ref{Jacobisfieldeq}.\\[-1ex]

Furthermore, by \eqref{tolltollinEins} and lemma \ref{gewuenschteUngl} we obtain
\begin{prop}\label{LewProp1}
 Let $\Mfg\subset\mathds{R}^{m+1}$ be an $m$-dimensional hypersurface of constant mean curvature. If for some fixed $I\in\left[\bigwedge_m\mathds{R}^{m+1}\right]$ the Gauss map satisfies $\Stufe{\I(x),I}{}>0$ for all $x\in\Mfg$, then $$ -\Delta_\Mfg \ln \Stufe{\I,I}{} \geq |\mathrm{B}|^2 \qquad\text{on $\Mfg$.}$$
\end{prop}

\begin{rem}
 In fact, an \textit{entire} constant mean curvature graph is minimal, cf. \cite{Chern2}, so, arguments as in \cite[sec. 5]{Fis} yield Moser's Bernstein theorem \cite{Moser}.
\end{rem}

\subsection{Intrinsic Geometry}\label{LewIntrinsicGeo}
More relations to classical differential geometry can be established if we reduce our observations of the derivatives to their \textit{tangential} parts. In particular, one can find the exterior and the adjoint derivative, the Dirac operator, as well as the Hodge-de Rham-Laplacian, and, of course, all the relations between them, known from standard monographs like \cite{LM} and \cite{Jost}. These have been partially deduced in \cite[sec. 4-3, ch 6]{HS}. Note that the tangential parts of the derivatives of the Gauss map do not, however, contain the crucial information used in the above discussions, e.g.
\begin{equation}
  \nabla \I \equiv 0,
\end{equation}
cf. also \cite[(5.25)]{LM}, and thus, the beautiful relation to the mean curvature vector field gets lost. For the second derivative it follows
\begin{equation}
 \nabla^2 \I \equiv 0.
\end{equation}
The operator $\nabla^2$ is called `colaplacian' in \cite{HS}. Indeed, it corresponds to the well known Hodge-de Rham-Laplace operator, cf. also \cite[(4-3.13)]{HS}. But the lack of the normal components would not yield the above theorems.


\end{document}